\documentclass[12pt]{amsart}
\usepackage{amsmath,amssymb,amsfonts,amsthm,amsopn}
\usepackage{latexsym,graphicx}
\usepackage{xcolor}
\usepackage{color, colortbl}
\usepackage{pb-diagram}
\usepackage[title]{appendix}

\newcommand{\tf}{time-frequency}

\newcommand{\tfs}{time-frequency shift}

\newtheorem{theorem}{Theorem}[section]

\newtheorem{corollary}[theorem]{Corollary}
\newtheorem{proposition}[theorem]{Proposition}
\newtheorem{definition}[theorem]{Definition}

\newtheorem{example}[theorem]{Example}

\newcommand{\beqa}{\begin{eqnarray*}}
\newcommand{\eeqa}{\end{eqnarray*}}

\newcommand{\field}[1]{\mathbb{#1}}
\newcommand{\bR}{\field{R}}        
\newcommand{\bN}{\field{N}}        
\newcommand{\bC}{\field{C}}        
\def\G{\mathcal{G}}

\def\la{\lambda}

\def\cF{\mathcal{F}}              
\def\cS{\mathcal{S}}

\def\cG{\mathcal{G}}

\def\cA{\mathcal{A}}

\def\cC{\mathcal{C}}

\def\a{\aleph}

\def\rd{\bR^d}

\def\rdd{{\bR^{2d}}}

\def\lrd{L^2(\rd)}

\def\intrd{\int_{\rd}}
\def\intrdd{\int_{\rdd}}

\def\R{\right)}

\def\<{\left<}
\def\>{\right>}

\def\mv1{M_v^1}

\def\Mmpq{M_m^{p,q}}
\def\phas{(x,\xi )}
\def\mn{(m,n)}
\def\mn'{(m',n')}

\def\o{\xi}
\def\a{\alpha}

\def\R{\mathbb{R}}
\def\Ren{\mathbb{R}^d}

\def\sch{\mathcal{S}}

\def\f{\varphi}

\def\Sn2{S_{2}(L^{2}(\Ren))}
\def\S1{S_{1}(L^{2}(\Ren))}
\def\sig00{\sigma_{0,0}}

\def\la{\langle}
\def\ra{\rangle}

\begin{document}

\begin{abstract} We exhibit the connection between the Wigner kernel and the Gabor matrix of a linear  bounded operator $T:\cS(\rd)\to\cS'(\rd)$. The smoothing effect of the Gabor matrix is highlighted by basic examples. This connection allows a comparison between the classes of Fourier integral operators  defined by means of the Gabor matrix in \cite{CGNRJMPA} and the Wigner kernel in \cite{CRGVFIO2}, showing the nice off-diagonal decay of the Gabor class with respect to the Wigner kernel one and suggesting further investigations. Modulation spaces containing the Sj\"ostrand class are the symbol classes of this study.
\end{abstract}

\title{Wigner kernel and Gabor matrix of operators}

\author{Elena Cordero}
\address{Universit\`a di Torino, Dipartimento di Matematica, via Carlo Alberto 10, 10123 Torino, Italy}
\email{elena.cordero@unito.it}
\author{Gianluca Giacchi}
\address{Universit\'a di Bologna, Dipartimento di Matematica, Piazza di Porta San Donato 5, 40126 Bologna, Italy; University of Lausanne, Switzerland; HES-SO School of Engineering, Rue De L'Industrie 21, Sion, Switzerland; Centre Hospitalier Universitaire Vaudois, Switzerland}
\email{gianluca.giacchi2@unibo.it}
\author{Luigi Rodino}
\address{Universit\`a di Torino, Dipartimento di Matematica, via Carlo Alberto 10, 10123 Torino, Italy}
\email{luigi.rodino@unito.it}

\thanks{}

\maketitle

\section{Introduction}

The Gabor matrix of an operator was introduced by Gr\"ochenig in the pioneering work  \cite{charly06} for the study of pseudodifferential operators and later extended to Fourier integral operators in \cite{CGNRJMP2014}. These two papers paved the way to many subsequent articles with applications to PDE's and Quantum Mechanics. The contributions are so many that it is impossible to cite them all (cf. the textbook \cite{Elena-book} for a partial list). 
Recently, evolution equations and dynamical version of Hardy uncertainty principles \cite{Helge} have suggested to look at the Wigner kernel of an operator, introduced and studied in \cite{CRGFIO1,CRGVFIO2}.  We want here to establish a connection between it and the Gabor matrix.

To introduce these features properly, we need to present first some basic elements of \tf\, analysis.

Given  $z=(x,\xi)\in\rdd$, we define the related  \tfs \,
acting on a function or distribution $f$ on $\rd$ as
\begin{equation}
	\label{eq:kh25}
	\pi (z)f(t) = M_\xi T_xf(t)=e^{2\pi i \xi t} f(t-x), \, \quad t\in\rd.
\end{equation}

 The short-time Fourier
transform (STFT) $V_gf$ of a
function/tempered distribution $f$ in $\cS'(\rd)$ with
respect to the the window $g\in\cS(\rd)\setminus\{0\}$ is defined by
\[
V_g f(x,\o)=\la f, M_{\o}T_xg\ra =\int e^{-2\pi i \o y}f(y)\overline{g(y-x)}\,dy,
\]
(i.e.,  the  Fourier transform $\cF$
applied to $f\overline{T_xg}$). \par
The STFT enjoys the following inversion formula \cite[Thrm. 1.2.16]{Elena-book}: assume $g,\gamma\in L^2(\rd)$, with $\la g,\gamma\ra\not=0$. Then, for all $f\in\lrd$, in terms of vector-valued integrals,
\begin{equation}\label{invformula}
	f=\frac1{\la \gamma,g\ra}\int_{\R^{2d}} V_g f(z) \pi (z)  \gamma\, dz \, .
\end{equation}
 The adjoint operator of $V_\gamma$,  defined by
$$V_\gamma^\ast F(t)=\intrdd F(z)  \pi (z) \gamma dz \, ,\quad F\in L^2(\rdd),
$$
maps  $L^{2}(\rdd)$ into $L^2(\rd)$. In particular, if $F=V_g f$ the inversion formula \eqref{invformula} can be rephrased as
\begin{equation}\label{treduetre}
	{\rm Id}_{L^2}=\frac 1 {\la \gamma,g\ra} V_\gamma^\ast V_g.
\end{equation}
\begin{definition}\label{Gmatrix}
Fix  $g,\gamma\in \cS(\rd)\setminus \{0\}$. The \emph{Gabor matrix} of a linear continuous operator $T$ from $\cS(\rd)$ to $\cS'(\rd)$ is defined to be 
\begin{equation}\label{unobis2s} \langle T \pi(z)
	g,\pi(w)\gamma\rangle,\quad z,w\in \rdd.
\end{equation}
\end{definition}
The Gabor matrix can be viewed as  a kernel of an integral operator. For simplicity, choose $g,\gamma\in\cS(\rd)$ such that $\la \gamma,g\ra=1$ and 
the inversion formula \eqref{treduetre} becomes  $V_\gamma^\ast V_g={\rm Id}$ (or, switching $g$ and $\gamma$, $V_g^\ast V_\gamma={\rm Id}$). A linear and bounded operator $T:\cS(\rd)\to\cS'(\rd)$ can  be written as
\begin{equation}\label{e8}
	T=V_\gamma^\ast V_g T V_\gamma^\ast V_g\,.
\end{equation}
The linear transformation $ V_g T V_\gamma^\ast$ is an integral operator whose kernel coincides with the Gabor matrix of $T$:
$$K_{T,g,\gamma}(w,z)=\la T\pi(z)g,\pi(w)\gamma\ra,\quad\,\,w,z\in\rdd. 
$$

Estimates in the sequel will not depend on the choice of $g,\gamma$, and one can limit attention to the case $g=\gamma$, cf. \cite{Elena-book}.
For sake of generality, we shall work with the former case, viewing the latter as special case when $\gamma=g$.

The Gabor matrix approach was successfully used to characterize algebras of pseudodifferential operators by Gr{\"o}chenig in \cite{charly06} and later with Rzeszotnik  in \cite{GR}:
\begin{equation}\label{6bis}
	\sigma^wf(x)=\intrdd e^{2\pi i x \xi}\sigma\left(\frac{x+y}{2},\xi\right) dyd\xi.
\end{equation}
The  symbol class object of their investigation is the so-called Sj\"ostrand class or modulation space
\begin{equation}\label{quattro}
	S_w=M^{\infty,1}(\rdd),
\end{equation}
of symbols $\sigma$ such that, for a fixed window $G\in\cS(\rdd)\setminus\{0\}$,
 $$\int_{\rd}\sup_{z\in\rd}|\langle \sigma ,\pi(z,\zeta )G\rangle|
\,d\zeta<\infty.$$
The following related scale of spaces were considered, too:
\begin{equation}\label{cinque}
	S^s_w=M^{\infty,\infty}_{1\otimes v_s}(\rdd),\quad v_s(z)=(1+|z|^2)^{s/2}, \,\,z\in\rdd,
\end{equation}
of  symbols $\sigma$ satisfying
$$\sup_{z,\zeta\in\rd}|\langle \sigma ,\pi(z,\zeta )G\rangle| v_s(\zeta)
\,<\infty,$$
with  the  parameter $s\in [0,\infty )$, see details in the next Section $2.1$.\par

Notice that the regularity of the spaces $S^s_w$ increases with $s$
 whereas in the
maximal space $S_w$ even  differentiability is lost \cite{KB2020,Elena-book,book}. Our attention in the sequel will be focused on the scale \eqref{quattro}.

The  characterization of pseudodifferential operators was further extended to Fourier integral operators (FIOs) of Schr\"{o}dinger type in \cite{CGNRJMPA}, constructing Wiener subalgebras
of FIOs with symbols in $S^s_w$. This paper paved the way to many other contributions in this framework, addressing regularity properties of FIOs by means of the off-diagonal decay of their Gabor matrix, cf. \cite[Chapter 5]{Elena-book} for a partial list of references.  

Precisely, a class of FIOs associated to a canonical transformation $\chi$ (cf. Definition \ref{def2.1} below) was constructed as follows (we refer to Section $2$ for the properties of $\chi$).

\begin{definition}\label{defFIO}
	Let $\chi$ be a  canonical transformation  satisfying {\it B1} and {\it B2} in Definition \ref{def2.1}, and $s\geq0$. Fix $g\in\cS(\rd)\setminus\{0\}$.  We say that  a
	continuous linear operator $T:\cS(\rd)\to\cS'(\rd)$ is in the
	class $FIO_\G(\chi,s)$, if its Gabor matrix  satisfies the decay
	condition
	\begin{equation}
		\label{unobis2a}
		|\langle T \pi(z) g,\pi(w)g\rangle|\leq {C}\langle w-\chi(z)\rangle^{-s},\qquad \forall\, w,z\in \rdd.
	\end{equation}
\end{definition}
 If the study is limited to FIOs of type I with symbol $\sigma\in\cS'(\rdd)$ and \emph{tame phase} $\Phi$ (Definition \ref{def2.1}), namely to   operators formally written as 
 \begin{equation}\label{sei}
 	T f(x)=T_{I,\Phi,\sigma}f(x)=\int_{\rd} e^{2\pi i
 		\Phi(x,\xi)}\sigma(x,\xi)\widehat{f}(\xi)\,d\xi \, ,\quad f\in\cS(\rd),
 \end{equation}
then the characterization involves the symbol classes $S^s_w$ as follows:
\begin{theorem}[\cite{CGNRJMP2014,Elena-book}]\label{caraI}
 Fix $g\in\cS(\rd)\setminus\{0\}$ and $s\geq 0$.
	Let $T$ be a continuous linear operator $\cS(\rd)\to\cS'(\rd)$ and $\Phi=\Phi_\chi$ 
	 a \emph{tame} phase function associated to the canonical transformation $\chi$, see the next Definition \ref{defFIO}.  Then the following properties are
	equivalent. \par {\rm (i)} $T=T_{I,\Phi_\chi,\sigma}$ is a FIO of type
	I for some $\sigma\in S^s_w$. \par {\rm
		(ii)} $T\in FIO_\G(\chi ,s)$. 
\end{theorem}
In other words, $T=T_{I,\Phi_\chi,\sigma}$ is a FIO of type
I for some $\sigma\in S^s_w$ if and only if the Gabor matrix satisfies the off diagonal decay
	\begin{equation}\label{unobis}
		|\langle T \pi(z) g,\pi(w)g\rangle|\leq {C}\langle w-\chi(z)\rangle^{-s},\qquad  w,z\in \rdd.
	\end{equation}

For the pseudodifferential operator $T=\sigma^w$ these estimates were proved in \cite{charly06,GR} with $\chi=I$, the identity operator.

The Gabor matrix of an operator $T$ has a drawback, which lies inside its definition and cannot be overcome: it depends on the window $g$ used for its construction. The nice off-diagonal decay is due to them, obviously. This concern has led to the search for a \emph{\tf}  kernel of $T$ which displays the nature of the operator without   contaminations from additional window functions.

 Inspired by the original work by Wigner \cite{Wigner}, in the recent contributions \cite{CRGFIO1,CRGVFIO2} we replaced  the Gabor matrix by the Wigner kernel of an operator.

\begin{definition} Consider $f,g\in\lrd$. 
	The (cross-)Wigner distribution is the time-frequency representation defined by
	\begin{equation}\label{CWD}
		W(f,g)(x,\xi)=\intrd f(x+\frac t2)\overline{g(x-\frac t2)}e^{-2\pi i t\xi}\,dt,\quad \phas\in\rdd.
	\end{equation} If $f=g$ we write $Wf:=W(f,f)$, the so-called Wigner distribution of $f$.
\end{definition}
Wigner  used it to analyse the action of  Schr\"odinger propagators.
In \cite{CRGFIO1} we extended  his  approach as follows.
\begin{definition}[Wigner Kernel]
Let $T$ be a continuous linear operator $\cS(\rd)\to \cS'(\rd)$ and define $K$ the operator such that 
	\begin{equation}\label{I3}
		W(Tf,Tg) = KW(f,g), \qquad f,g\in\cS(\rd).
	\end{equation}
	Its  integral kernel $k$ is called the \emph{Wigner kernel}  of $T$. Namely,
	 $k$ is the distribution in $\cS'(\bR^{4d})$  satisfying
		\begin{equation}\label{kerFormula}
			\la W(Tf,Tg),W(u,v)\ra=\la k,W(u,v)\otimes\overline{W(f,g)}\ra, \qquad f,g,u,v\in\cS(\rd).
		\end{equation}
\end{definition}
	This implies the integral formula 
	\begin{equation}\label{I4}
		W(Tf,Tg)(z) = \intrdd k(z,w) W(f,g)(w)\,dw,\quad z\in\rdd,\quad f,g\in\cS(\rd).
	\end{equation}

The operator  $K$ is a well-defined continuous linear operator $\cS(\rdd)\to \cS'(\rdd)$.
Notice that it does not contain any windows in its definition and the Wigner kernel depends only on the Schwartz integral kernel of $T$. In fact, define $\mathfrak{T}_pF(x,\xi,y,\eta)=F(x,y,\xi,-\eta)$, $x,\xi,y,\eta \in\rd$, then

\begin{proposition}\label{3.3}
	Consider $T$ as above and let $k_T\in\cS'(\rdd)$ be its Schwartz integral kernel. Then, there exists a unique distribution $k\in\cS'(\bR^{4d})$ such that \eqref{kerFormula} holds. Hence, every bounded linear operator $T:\cS(\rd)\to\cS'(\rd)$ has a unique Wigner kernel $k$. Furthermore,
	\begin{equation}\label{nuclei}
		k=\mathfrak{T}_pWk_T.
	\end{equation}
\end{proposition}

For the proof of this proposition we address to \cite{CRGFIO1}.

Inspired by the FIOs classes $FIO_\cG(\chi,s)$ in Definition \ref{defFIO}, we introduced in \cite{CRGVFIO2} the class of FIOs FIO($\chi$, $N$) as follows.

\begin{definition}\label{C6T1.1}
	Consider $\chi$ as in the preceding Definition \ref{defFIO}.  For  $N\in \bN_+$,   we say that the operator $K$ in \eqref{I3} is in the class FIO($\chi$, $N$) if its Wigner kernel $k$ in \eqref{I4} satisfies
	\begin{equation}\label{nucleoFIO}
		|k(z,w) |\lesssim \frac{1}{\la z-\chi(w)\ra^{N}},\quad w,z\in\rdd.
	\end{equation}
\end{definition}

The goal of this paper is to exhibit the connection between the Gabor matrix and the Wigner kernel $k$ of an operator and, consequently, to show the relation between the related classes of FIOs which arise from them. 

Our first result is the following:
\begin{theorem}\label{3.4}
	Fix $g,\gamma\in\cS(\rd)$ such that $\la g,\gamma\ra\not=0$ and $T$ a continuous linear operator $\cS(\rd)\to \cS'(\rd)$. Then,
	\begin{equation}\label{E3}
		|\la T\pi(z)g,\pi(w)\gamma\ra|^2=[k\ast (W\gamma\otimes Wg)](w,z),\quad\forall z,w\in\rdd.
	\end{equation}
\end{theorem}
Using  the equality above  we will prove the  inclusion:
\begin{equation}\label{inclusion}
	FIO(\chi,N)\subset  FIO_\G(\chi,N/2), 
\end{equation}
so for $FIO_\G(\chi,N)$ half of the  decay is lost.

In particular, taking the Gaussian function $\f(t)=e^{-\pi t^2}$ its  Wigner distribution is \cite[Lemma 1.3.12]{Elena-book} $W\f(x,\xi)=2^{d/2}e^{-2\pi (x^2+\xi^2)}$ so that  
writing
$$\Phi(w,z)=W\f(w)W\f(z)=2^{d} e^{-2\pi(w^2+z^2)}, \quad w,z\in\rdd.$$
we have the connection
\begin{equation}\label{E4}
	|\la T\pi(z)\f,\pi(w)\f\ra|^2=[k\ast \Phi](w,z), \quad w,z\in\rdd.
\end{equation}

\begin{proposition}\label{E1}
	For $N>2d$, consider $T\in FIO(\chi,N)$ with $\chi$ \emph{tame} canonical transformation. Then $T$ can be represented as a type I FIO $T=T_{I,\Phi_\chi,\sigma}$ with symbol $\sigma\in S_w^{N/2}$. 
\end{proposition}

The opposite inclusion in \eqref{inclusion} is not true, however we have a partial converse.

\begin{theorem}\label{partialconverse}
	Consider a type I FIO $T=T_{I,\Phi_\chi,\sigma}$ with $\sigma\in S^N_w$. Let $k$ be the Wigner kernel of $T$ and 
	\begin{equation}\label{Gaussian}
		\Phi(w,z)=2^d e^{-2\pi(w^2+z^2)},\quad w,z\in\rdd.
	\end{equation}
Define $\tilde{k}$ to be the smoothed Wigner kernel:
\begin{equation}\label{tildek}
	\tilde{k}=k\ast \Phi.
\end{equation}
Then
\begin{equation}
	\tilde{k}(w,z)\lesssim \frac{1}{\la z-\chi(w)\ra^{2N}},\quad w,z\in\rdd.
\end{equation}
\end{theorem}

The next Section $2$ contains some preliminaries. Section $3$ is devoted to the proofs of the results exhibited above. In Section $4$ we give some examples which clarify the smoothing effect of the Gabor matrix in \eqref{E4}.

\section{Preliminaries}
\textbf{Notation.} We define $t^2=t\cdot t$,  $t\in\rd$, and, similarly,
$xy=x\cdot y$.  The space   $\sch(\Ren)$ is the Schwartz class and $\sch'(\Ren)$ its dual (the space of temperate distributions).   The brackets  $\la f,g\ra$ means the extension to $\sch' (\Ren)\times\sch (\Ren)$ of the inner product $\la f,g\ra=\int f(t){\overline {g(t)}}dt$ on $L^2(\Ren)$ (conjugate-linear in the second component).

$GL(d,\R)$ denotes the group of real invertible $d\times d$ matrices. 

\subsection{Modulation spaces \cite{Elena-book,F1,B36}} Let us add few details about function spaces. As stated in the Introduction, in our study we shall consider the following weight functions
\begin{equation}\label{vs} v_s(z)=\la z\ra^s=(1+|z|^2)^{\frac s 2},\quad s\in\R.
\end{equation}

The most suitable symbol spaces for time-frequency analysis are the modulation spaces.  Introduced by Feichtinger in the 80's (see the original paper \cite{F1}) and now are well-known in the framework of time-frequency analysis \cite{Elena-book,book}.
Here we limit to the case of weighted modulation spaces with
 weight functions $m$ of at most  polynomial growth, basic examples are \eqref{vs}.  
 
 Let $g$ be a non-zero Schwartz  function. For $1\leq p,q \leq \infty$ and $m$ a weight function as before,   the modulation space $M ^{p,q}_{m}(\R^d)$ is the space of
 distributions $f\in\cS'(\rd)$ such that their STFT $V_gf$ belongs to the
 space $L^{p,q}_{m}(\rdd ) $ with  norm
 \[
 \|f\|_{M ^{p,q}_{m}(\R^d)}:=\|V_gf\|_{L^{p,q}_{m}(\R^{2d})}=\left(\intrd\left(\intrd|V_gf\phas|^p m\phas^p dx\right)^{\frac q p}d\xi\right)^{\frac1q}.
 \]
 This definition  does not depend on the choice of
 the window $g\in \cS (\rd ), g \neq 0$, and different windows yield
 equivalent norms on $\Mmpq$~\cite[Thm.~11.3.7]{book}. 
 
 The symbol spaces we shall  be mainly concerned with are
 $S^s_w=M^{\infty,\infty}_{1\otimes v_s}(\rdd)$ with the  norm
 $$\|\sigma\|_{S^s_w}=\sup_{z\in\rdd}\sup_{\zeta\in\rdd} |V_G\sigma|
 (z,\zeta)|\,\langle \zeta \rangle ^s,
 $$
for a fixed $G\in\cS(\rdd)\setminus\{0\}$.

\subsection{Tame phase functions and related canonical transformations}
\begin{definition}\label{def2.1} Following the notation of \cite{CRGFIO1,CGNRJMPA}, a real  phase function  $\Phi$  is named \emph{tame} if it satisfies the following
	properties:\\
	{\it  A1.} $\Phi\in \cC^{\infty}(\rdd)$;\\
	{\it  A2.} For $z\in\rdd$,
	\begin{equation}\label{phasedecay}
		|\partial_z^\a \Phi(z)|\leq
		C_\a,\quad |\a|\geq
		2;\end{equation}
	{\it  A3.} There exists $\delta>0$:
	\begin{equation}\label{detcond}
		|\det\,\partial^2_{x,\eta} \Phi(x,\eta)|\geq \delta.
	\end{equation}
	\par
	Solving the system
	\begin{equation}\label{cantra} \left\{
		\begin{array}{l}
			y=\Phi_\eta(x,\eta),
			\\
			\xi=\Phi_x(x,\eta), \rule{0mm}{0.55cm}
		\end{array}
		\right.
	\end{equation}
	with respect to $(x,\xi)$, one obtains a
	map $\chi$
	\begin{equation}\label{chi}
		(x,\xi)=\chi(y,\o),
	\end{equation}
	with the following properties:\\
	
	\noindent {\it  B1.} $\chi:\rdd\to\rdd$ is a \emph{symplectomorphism} (smooth, invertible,  and
	preserves the symplectic form in $\rdd$).
	\\
	{\it  B2.} For $z\in\rdd$,
	\begin{equation}\label{chistima}
		|\partial_z^\a \chi(z)|\leq C_\a,\quad\mbox{for} \quad|\a|\geq 1;\end{equation}
	{\it B3}. There exists $\delta>0$: 
	\begin{equation}\label{detcond2}
		|\det\,\frac{\partial x}{\partial y}(y,\eta)|\geq \delta,\quad\,\,
		(x,\xi)=\chi(y,\eta).
	\end{equation}
\end{definition}
Conversely,   to every transformation $\chi$ satisfying the three hypothesis above corresponds a tame phase $\Phi$, uniquely
determined up to a constant \cite{CGNRJMPA}. 
\section{Properties of the Wigner Kernel}
In what follows we showcase the connection between the Wigner kernel of an operator $T$ and its Gabor matrix, namely we shall prove Theorem \ref{3.4} in the Introduction.

\begin{proof}[Proof of Theorem \ref{3.4}]
	In the following computations we will apply the definition of Wigner kernel and the covariance of the Wigner distribution:
	\begin{align}
		|\la T\pi(z)g,\pi(w)\gamma\ra|^2&=\la T\pi(z)g,\pi(w)\gamma\ra\overline{\la T\pi(z)g,\pi(w)\gamma\ra}\\
		&=\la T\pi(z)g\otimes\overline{T\pi(z)g},\pi(w)\gamma\otimes\overline{\pi(w)\gamma}\ra\\
		&=\la W(T\pi(z)g),W(\pi(w)\gamma)\ra\notag\\
		&=\la k, W(\pi(w)\gamma)\otimes W(\pi(z)g)\ra\notag\\
		&=\int_{\bR^{4d}}k(u,v)W(\pi(w)\gamma)(u)W(\pi(z)g)(w)dudv\notag\\
		&=\int_{\bR^{4d}}k(u,v)W\gamma(u-w)Wg(v-z)dzdv\label{E2}\\
		&=[k\ast (W\gamma\otimes Wg)](w,z),\quad\forall z,w\in\rdd.
		\end{align}
	The last equality yields  the claim.
\end{proof}

Next, we prove the relationship between the two classes of FIOs.
\begin{theorem}\label{conn}
	Consider $T\in FIO(\chi,N)$  and windows $g,\gamma\in\cS(\rd)\setminus\{0\}$. Then, for $N>2d$, the associated Gabor matrix of $T$ satisfies the off-diagonal decay estimate: 
	\begin{equation}\label{Gabordecay}
		|\la T\pi(z)g,\pi(w)\gamma\ra|\lesssim \frac{1}{\la w-\chi(z)\ra^{N/2}}, \qquad w,z\in\rdd.
	\end{equation}
\end{theorem}
\begin{proof}
	From \eqref{E2} above we have the equality
	$$|\la T\pi(z)g,\pi(w)\gamma\ra|^2=\int_{\bR^{4d}}k(u,v)W\gamma(u-w)Wg(v-z)dzdv.$$
	 Next, recall the bi-Lipschitz property of $\chi$, which yields the estimate  $|v-z|\asymp |\chi(v)-\chi(z)|$. Furthermore, we will perform the change of variables $\chi(v)-\chi(z)=v'$ so that $dv=|\det J\chi^{-1}(v)| dv'$  and
	$|\det J\chi_1^{-1}(v)|\leq C$ by \eqref{chistima}. Namely,
	\begin{align*}
		|\la T\pi(z)g,\pi(w)\gamma\ra|^2&\leq\int_{\bR^{4d}}|k(u,v)||W\gamma(u-w)||Wg(v-z)|dudv
		\\
		&\lesssim\int_{\bR^{4d}}\frac{1}{\la u-\chi(v)\ra^{N}}\frac{1}{\la u-w \ra^{N}}\frac{1}{\la v-z \ra^{N}}dudv\\
		&=\int_{\bR^{4d}}\frac{1}{\la z+w -\chi(v)\ra^{N}}\frac{1}{\la u \ra^{N}}\frac{1}{\la v-z \ra^{N}}dzdw\\
		&\asymp\int_{\bR^{4d}}\frac{1}{\la u+w -\chi(v)\ra^{N}}\frac{1}{\la u \ra^{N}}\frac{1}{\la \chi(v)-\chi(z) \ra^{N}}dudw\\
		&\leq C\int_{\bR^{4d}}\frac{1}{\la u+w -\chi(z)-v'\ra^{N}}\frac{1}{\la u \ra^{N}}\frac{1}{\la v' \ra^{N}}dudv'\\
	\end{align*}
	Using the weight convolution property  \cite[Lemma 11.1.1]{book}: $$\la\cdot\ra^{-N}\ast\la\cdot\ra^{-N}\leq\la\cdot\ra^{-N}\quad\mbox{for}\quad N>2d,$$
	\begin{align*}
		|\la T\pi(z)g,\pi(w)\gamma\ra|^2&\lesssim\int_{\rdd}\frac{1}{\la u+w -\chi(z)\ra^{N}}\frac{1}{\la u\ra^{N}}dz\\
		&=\int_{\rdd}\frac{1}{\la\chi(z)-w-u\ra^{N}}\frac{1}{\la u\ra^{N}}du\\
		&\lesssim \frac{1}{\la\chi(z)-w\ra^{N}}=
		\frac{1}{\la w-\chi(z)\ra^{N}}
	\end{align*}
	by the even property of $\la\cdot\ra$.
	Taking the square root of the above inequality we obtain \eqref{Gabordecay}.
\end{proof}
\begin{corollary}\label{cor3.3}
	If $T\in FIO(\chi,N)$, $N>2d$, then  $T\in FIO_\G(\chi,N/2)$.
\end{corollary}
\begin{proof}
	It follows from Definition \ref{defFIO}, Definition \ref{C6T1.1}, and the estimate \eqref{Gabordecay}.
\end{proof}

The result above gives:
$$ FIO(\chi,N)\subset  FIO_\G(\chi,N/2),$$  
which implies for $FIO_\G(\chi,N)$ that half of the  decay is lost. 

The inclusion above allows to prove Proposition \ref{E1}:

\begin{proof}[Proof of Proposition \ref{E1}]
	It follows from Corollary \ref{cor3.3} and the characterization in Theorem \ref{caraI}.
\end{proof}

To ensure a better decay of the Wigner kernel, we undergo it to the smoothing process illustrate in Theorem \ref{partialconverse} in the Introduction.

\begin{proof}[Proof of Theorem \ref{partialconverse}]
	Using formula \eqref{E4} we infer
	$$\tilde{k}(w,z)=|\la T\pi(z)\f,\pi(w)\f\ra|^2.$$
	From Theorem \ref{caraI} and \eqref{unobis} it follows $T\in FIO_\cG(\chi,N)$, hence it satisfies the (continuous) estimate
	$$|\langle T \pi(w) \f,\pi(z)\f\rangle|\leq {C}\langle z-\chi(w)\rangle^{-N},\qquad  w,z\in\rdd.$$
	This implies
	$$ \tilde{k}(w,z)=|\la T\pi(z)\f,\pi(w)\f\ra|^2\lesssim \frac{1}{ \langle z-\chi(w)\rangle^{2N}},\qquad  w,z\in\rdd,$$
	as desired.
\end{proof}

The result in Theorem \ref{partialconverse}  suggests that we could replace the class $FIO(\chi,N)$ with the class $\widetilde{FIO}(\chi,N)$. Namely $T\in \widetilde{FIO}(\chi,N)$ if the smoothed Wigner kernel $\tilde{k}$ satisfies
\begin{equation}
	\tilde{k}(w,z)\lesssim \frac{1}{ \langle z-\chi(w)\rangle^{2N}}.
\end{equation}
We will pursue the study of this new class in a subsequent work.

\section{Examples of smoothing}\label{4}

Let us return to the identity \eqref{E4}:
\begin{equation}\label{E4-1}
	|\la T\pi(z)\f,\pi(w)\f\ra|^2=[k\ast\Phi](w,z), \qquad w,z\in\rdd,
\end{equation}
where
\begin{equation}\label{defPhi}
	\Phi(w,z)=2^de^{-2\pi(w^2+z^2)}.
\end{equation}
In general, the Wigner kernel of a bounded operator $T:\cS(\rd)\to\cS'(\rd)$ is a distribution $k\in\cS'(\bR^{4d})$, whereas the smoothing effect of the convolution in \eqref{E4-1} ensures that the Gabor matrix in the left hand-side of \eqref{E4-1} belongs to $\mathcal{C}^\infty(\bR^{4d})\cap\cS'(\bR^{4d})$. In fact, we may give examples where pointwise estimates have no meaning for $k$, because of its singularities, while the estimates \eqref{unobis2a} for the Gabor matrix are satisfied. Notably, this will show that the inclusion \eqref{inclusion} is strict.

\begin{example}[The Husimi distribution]\label{ex1}
	For the benefit of non-specialists, before discussing Wigner kernels in the following examples, we consider the Husimi distribution, \textcolor{black}{see e.g. \cite{Gosson-Wigner}}:
	\begin{equation}\label{Husimi}
		Wg\ast W\f = |V_\f g|^2, \qquad g\in\cS'(\rd),
	\end{equation}
	where, as in the previous sections, 
	\begin{equation}\label{44}
	\f(t)=e^{-\pi t^2}, \qquad and \qquad W\f(x,\xi)=2^{d/2}e^{-2\pi(x^2+\xi^2)}, \qquad t,x,\xi\in\rd. 
	\end{equation}
	This provides the basic example of smoothing for the Wigner transform. Let us test \eqref{Husimi} on $g=\delta$, the point measure $\la \delta,f\ra = \overline{f(0)}$. An easy calculation shows (see, e.g., \cite[Example 4.3.1]{book}) that
$$W\delta=\delta\otimes 1\in\cS'(\rdd).$$

Hence, $\delta$ has the Wigner time-frequency concentration displayed in Figure $1$.

\begin{figure}\label{figura1}
	\begin{center}
	\end{center}
	\begin{minipage}[h]{0.5\textwidth}
		\vspace{1cm}
		\includegraphics[width=1\textwidth]{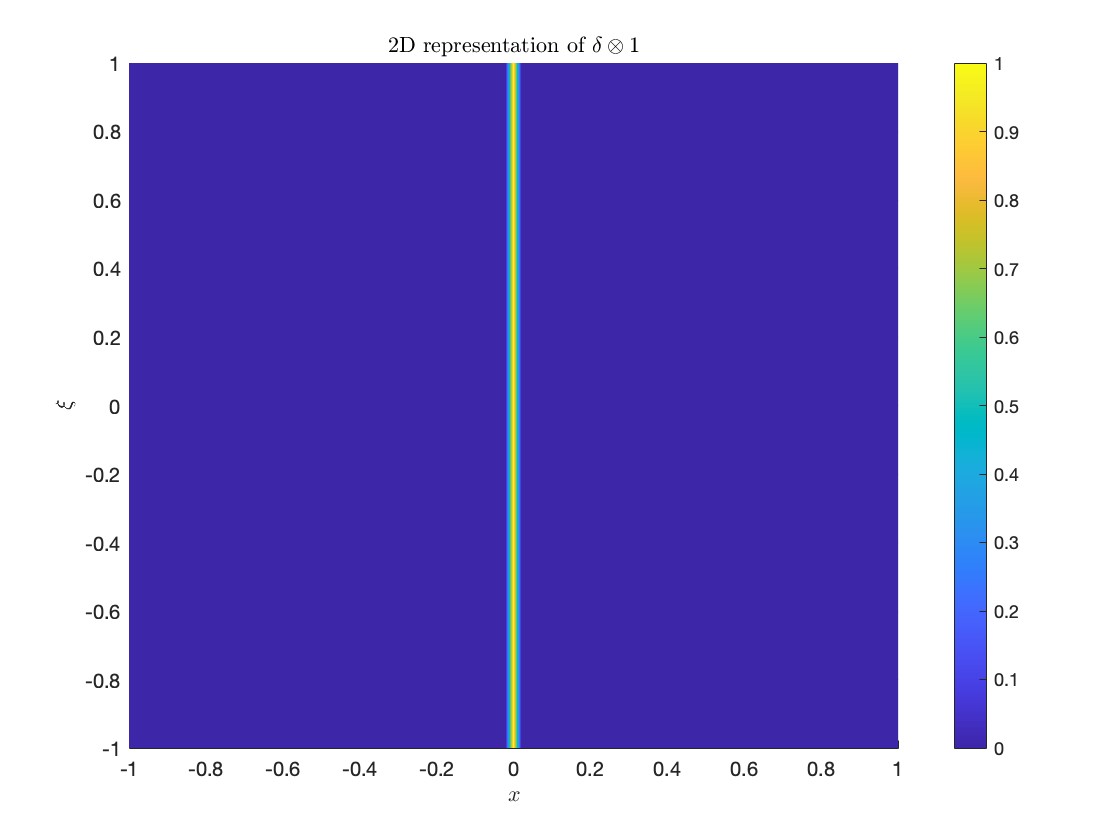}
	\end{minipage}\hfill
	\begin{minipage}[h]{0.5\textwidth}
		\vspace{1cm}\includegraphics[width=1\textwidth]{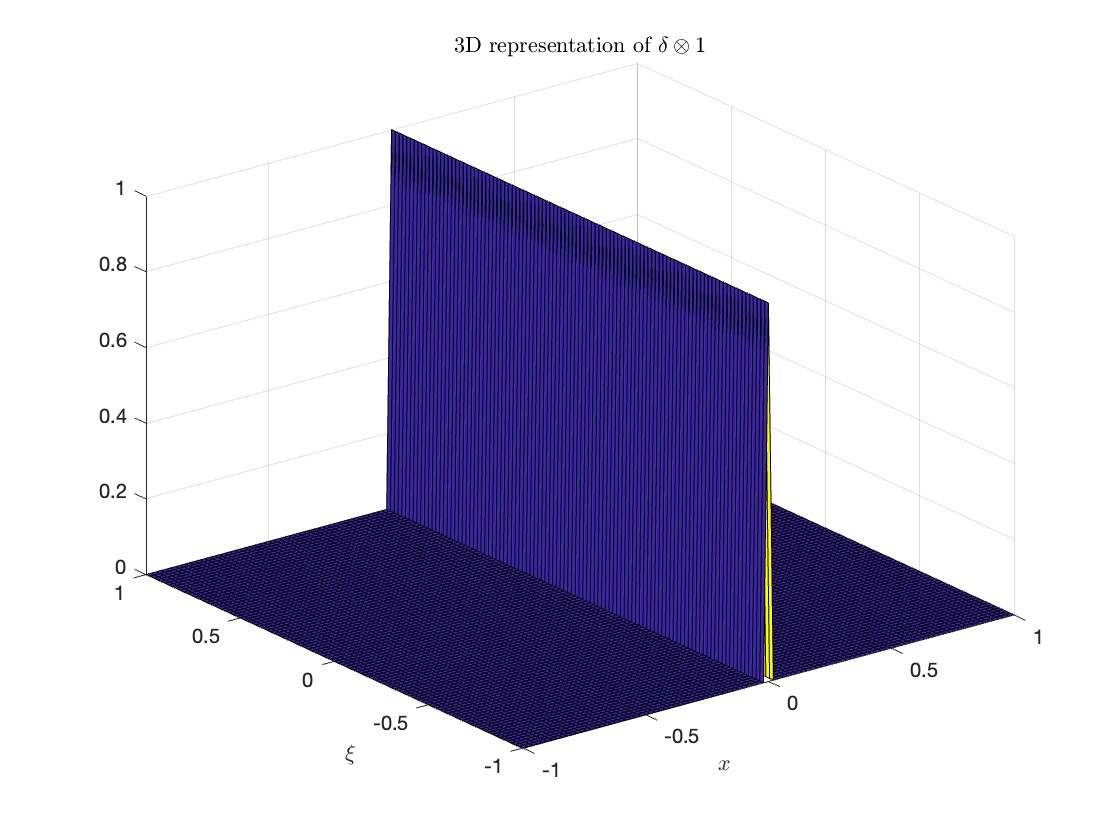}
	\end{minipage}\hfill
	\caption{2D and 3D representations of $\delta\otimes 1$.}
\end{figure}

From \eqref{Husimi},
$$W\delta\ast W\f(x,\xi)=V_\f \delta(x,\xi)\overline{V_\f \delta(x,\xi)},$$
and
$$V_\f \delta(x,\xi) =\la \delta,\pi(x,\xi)\f\ra=\overline{\f(-x)}=\f(x)$$
so that 
\[
	W\delta\ast W\f = (\delta\otimes 1)\ast W\f = 2^{d/2}e^{-2\pi x^2}\int_{\rd}e^{-2\pi\xi^2}d\xi=e^{-2\pi x^2}=|V_\f\delta|^2,
\]
as expected. Figure $2$ displays the time-frequency concentration of $\delta$.

\begin{figure}\label{Figura2}
	\begin{center}
	\end{center}
	\begin{minipage}[h]{0.5\textwidth}
		\vspace{1cm}
		\includegraphics[width=1\textwidth]{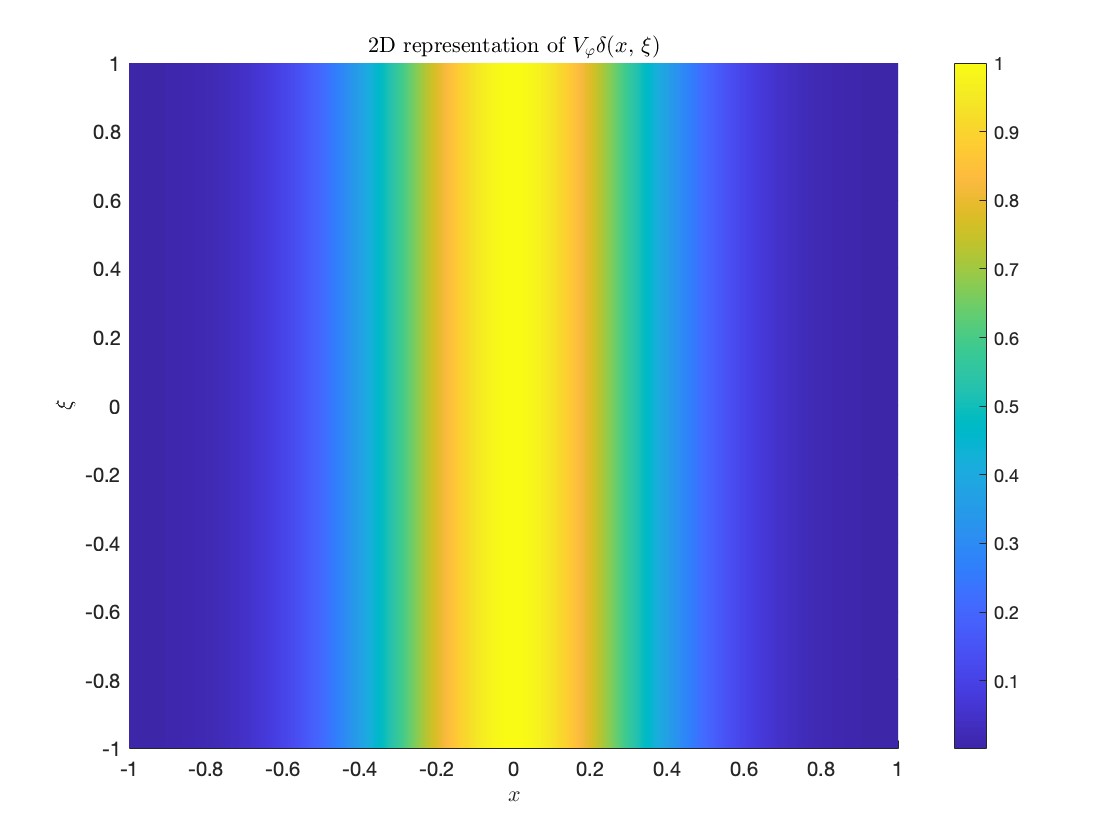}
	\end{minipage}\hfill
	\begin{minipage}[h]{0.5\textwidth}
		\vspace{1cm}\includegraphics[width=1\textwidth]{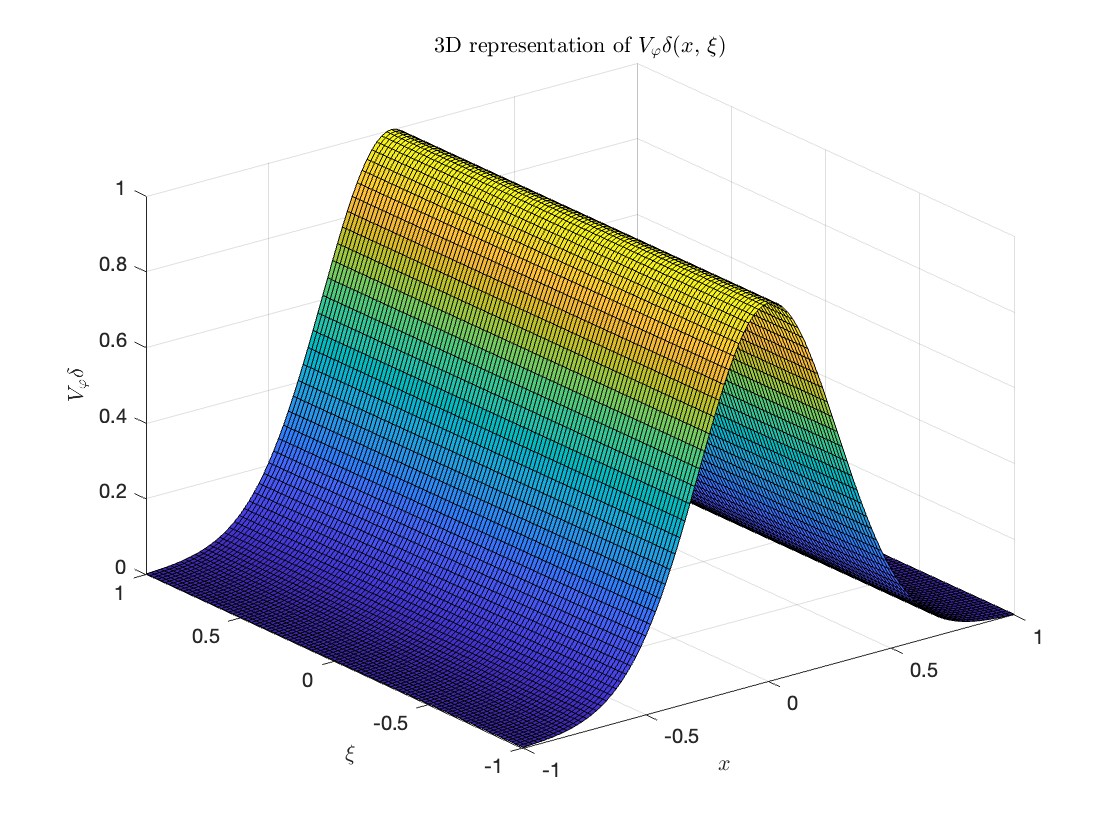}
	\end{minipage}\hfill
	\caption{2D and 3D representations of $V_\f\delta$.}
\end{figure}
\end{example}

\begin{example}[The identity operator]\label{ex2}

Consider $T=I$, the identity operator. Then
$$|\langle I \pi(z) g,\pi(w)g\rangle|=|V_{g} (\pi(z)g)(w)|=|V_{g} g(w-z)|\leq {C}\langle w-z\rangle^{-N},$$
for every $N\geq0$, hence $I\in  FIO_\G(I_{d\times d},N)$ ($I_{d\times d}$ being the identity matrix) for every $N\geq0$. However, from \eqref{I4},
$$W(If,Ig)(z)=W(f,g)(z)=\intrdd T_z\delta(w)W(f,g)(w)dw$$
so that  the Wigner kernel is the distribution $k(z,w)=\delta_{w-z}\in\cS'(\rdd)$ which does not belong to any $FIO(I_{d\times d},N)$. In fact, $\bigcup_\chi FIO(\chi,N)$ is an algebra for $N>2d$, as proved in \cite{CRGVFIO2}, but it is non-unital, whereas $\bigcup_\chi FIO_\cG(\chi,N)$ is an algebra with identity, and Wiener property, see \cite{CGNRJMPA}. To be definite, let us test the validity of \eqref{E4-1} for $g=\f$ as in \eqref{44}. We have,
\[
	V_\f\f(x,\xi)=2^{-d/2}e^{-i\pi x\xi}e^{-\pi(x^2+\xi^2)/2}, \qquad x,\xi\in\rd,
\]
hence:
\[
	|\la I\pi(z)\f,\pi(w)\f\ra|^2=2^{-d}e^{-\pi(z-w)^2}, \qquad z,w\in\rdd.
\]
On the other hand,
\begin{align*}
	[k\ast\Phi(w,z)]&=\delta_{w-z}\ast\Phi \\
	&=\int_{\bR^{4d}}\delta_{\vartheta-\zeta}\Phi(z-\zeta,w-\vartheta)d\theta d\zeta \\
	&= \int_{\rdd}\Phi(z-\vartheta,w-\vartheta)d\theta\\
	&=\int_{\rdd}\Phi(\lambda+\frac{1}{2}(z-w),\lambda-\frac{1}{2}(z-w))d\lambda,
\end{align*}
where we used the change of variables $\vartheta = -\lambda+\frac{1}{2}(z+w)$. Since from \eqref{defPhi}
\[
	\int_{\rdd}\Phi(\lambda+\frac{1}{2}z,\lambda-\frac{1}{2}z)d\lambda = W\Phi(z,0)=2^{-d}e^{-\pi z^2},
\]
we conclude:
\[
	[k\ast \Phi](w,z)=2^{-d}e^{-\pi(z-w)^2},
\]
and \eqref{E4-1} is satisfied. With respect to the off-diagonal variable $t=z-w\in\rdd$, we may note a formal identity with the expressions in Example \ref{ex1}.
\end{example}

\begin{example}[Metaplectic operators]\label{ex3}
	If $\cA$ is a symplectic matrix on $\rdd$, that is $\cA^T J\cA=J$, where
	\[
		J=\begin{pmatrix}
			0_{d\times d} & I_{d\times d}\\
			-I_{d\times d} & 0_{d\times d}
		\end{pmatrix}
	\]
	yields the standard symplectic form on $\rdd$, the metaplectic operator $\hat\cA$ is defined by the intertwining relation
	\[
	\pi(\cA z)=c_\cA\hat\cA\pi(z)\hat\cA^{-1}, \qquad z\in\rdd,
	\]
	with a phase factor $c_\cA\in\bC$, $|c_\cA|=1$. We denote by $Sp(d,\bR)$ the group of $2d\times 2d$ symplectic matrices and we refer to \cite{Gosson-Wigner} and \cite{book} for the theory of metaplectic operators. If $\cA\in Sp(d,\bR)$ and $g\in\cS(\rd)$, for every $N\geq0$ there exists a $C_N>0$ such that:
	\begin{equation}\label{46}
		|\la \hat\cA \pi(z)g,\pi(w)g\ra|\leq C_N\la w-\cA z\ra^{-N}, \qquad z,w\in\rdd,
	\end{equation}
	see for example \cite{CGNRJMP2014}. Hence, $\hat\cA\in FIO_\cG(\cA,N)$ for every $N$. Concerning the Wigner kernel, we note that:
	\[
		W(\hat\cA f)=Wf(\cA^{-1}z)=\int_{\rdd}\delta_{z=\cA w}Wf(w)dw,
	\]
	see for example \cite[Proposition 1.3.7]{Elena-book}, hence the Wigner kernel of $\hat\cA$ is given by:
	\[
		k(z,w)=\delta_{z=\cA w},
	\]
	distribution density in $\cS'(\bR^{4d})$. As in Example \ref{ex2}, pointwise estimates are not possible. By applying the smoothing in \eqref{E4-1}, we may recapture the Gabor matrix and the estimates \eqref{46}.
\end{example}

\begin{example}\label{ex4}
	A pseudodifferential operators $\sigma^w$ is defined as in \eqref{6bis}. Its Wigner kernels is computed in \cite{CRPartI2022,CGRPartII2022} and can be identified with the kernel $k(z,w)$ of a pseudodifferential operator with symbol in some class of distributions on $\bR^{4d}$. By using Theorem \ref{3.4} and, in particular, the smoothing \eqref{E4-1}, we may recapture the Gabor matrix and the results of \cite{charly06,GR}, i.e., $\sigma^w\in FIO_\cG(I_{d\times d},N)$. 
\end{example}

\begin{example}[Generalized metaplectic operators]\label{ex5}
	Generalized metaplectic operators can be defined as products:
	\[
		M_{\sigma,\cA}=\sigma^w\hat\cA,
	\]
	with $\sigma^w$ as in Example \ref{ex4} and $\hat\cA$ as in Example \ref{ex3}. Under suitable assumptions on $\cA$, $M_{\sigma,\cA}$ can be written as a FIO of type I, as in \eqref{sei}, with quadratic phase function and associated linear canonical transformation $\chi=\cA$. The Wigner kernel of $M_{\sigma,\cA}$ is given by $k(z,w)=h(z,\cA w)$, where $h$ is the Wigner kernel of $\sigma^w$. Again, by \eqref{E4-1}, we may recapture the Gabor matrix and prove that $M_{\sigma,\cA}\in FIO_\cG(\cA,N)$. At the moment, the characterization of the Wigner kernel for a FIO of type I with general nonlinear canonical transformation $\chi$ is still an open problem.
\end{example}

\end{document}